\documentclass[12pt]{article}

\usepackage{amsfonts,amsmath,amsthm,bm,bbm, thmtools, thm-restate}
\usepackage{latexsym,color,epsfig,mathrsfs,enumerate}
\usepackage[toc,page]{appendix}
\usepackage{fancyhdr}
\usepackage[colorlinks=true, allcolors=red]{hyperref}
\usepackage{indentfirst}
\usepackage[maxbibnames=9]{biblatex}
\usepackage{comment}
\usepackage{xcolor}
\usepackage[shortlabels]{enumitem}

\usepackage{graphicx}
\usepackage{pgfplots}
\pgfplotsset{compat=1.15}
\usetikzlibrary{arrows}

\nocite{*}

\newtheorem{thm}{Theorem}[section]
\newtheorem{dfn}[thm]{Definition}
\newtheorem{lem}[thm]{Lemma}

\newtheorem{prop}[thm]{Proposition}

\newtheorem{conjecture}[thm]{Conjecture}
\newtheorem{exam}{Example}

\usepackage{combelow}

\usepackage{newunicodechar}
\newunicodechar{ș}{\cb{s}}
\newunicodechar{ț}{\cb{t}}

\newcommand{\eps}{\varepsilon}

\newcommand{\R}{\mathbb R}

\title{Many Antipodal Pairs Force Many Neighboring Pairs}
\date{}

\author{G\'abor Dam\'asdi\thanks{HUN-REN Alfr\'ed R\'enyi Institute of Mathematics, ELTE Eötvös Loránd University, Budapest. Email: \texttt{damasdigabor@caesar.elte.hu}} \and Lauren\cb{t}iu Ploscaru\thanks{HUN-REN Alfr\'ed R\'enyi Institute of Mathematics, Budapest. Email: \texttt{laurentiuploscaru@renyi.hu}.
}}

\begin{document}

\maketitle

\begin{abstract}
	  Let $X=\{x_1,\dots,x_n\}\subset \R^2$ be a finite set of points of diameter at most $1$. It is natural to expect that if many pairs $(x_i,x_j)$ lie at distance close to $1$ from each other, then some clustering phenomenon must occur, implying that a significant number of these pairs are also very close to each other. 
      We prove that there exists a universal constant $c>0$ such that for all $0<\eps<1$, whenever $n$ is large enough, we have: 
\[
\big|\{i<j:\|x_i-x_j\|\le \eps\}\big|
\ge c\cdot \eps^{1/2}\cdot
\big|\{i<j:\|x_i-x_j\|\ge 1-\eps\}\big|.
\]
\indent This confirms a recent conjecture of Steinerberger, who asked whether the $\eps^{1/2}$ ratio is the best possible. We also study a two-parameter version of Steinerberger's question by considering the number of pairs at distance at most $\eps_1$ and at distance at least $1-\eps_2$. We show that in this case the optimal ratio is $\eps_1^2\cdot\eps_2^{-3/2}$ when $\eps_1\leq\eps_2$, and $\eps_1^{1/2}$ otherwise. The proof proceeds by introducing an auxiliary graph associated with the set $X$ and reducing the problem to bounding the largest eigenvalue of its adjacency matrix. Our main result is the outcome of human--AI interactions using ChatGPT 5.4. 
      
\end{abstract}

\section{Introduction}

 Consider a finite set of points  $X=\{x_1,\dots,x_n\}\subset \R^2$ of diameter at most 1. What can we say about the number of pairs of points that are within distance $\varepsilon$ from each other? The area of the convex hull is bounded, hence, provided $n\geq C\varepsilon^{-2}$ for a sufficiently large constant $C$, we have $\Omega(n^2\varepsilon^{2})$ such pairs, and this cannot be improved in general. In this paper we investigate how this behavior changes if we assume that there are many pairs of points in $X$ at distance at least $1-\varepsilon$. 

 Throughout, by a \emph{pair} $(x,y)$ we mean the $2$-element subset $\{x,y\}$, hence pairs are unordered and contain distinct points. For $0<\eps<1$, we call a pair $(x_i,x_j)$ $\varepsilon$-\emph{antipodal} if $\|x_i-x_j\|\ge 1-\eps$ and $\varepsilon$-\emph{neighboring} if $\|x_i-x_j\|\le \eps$. We call the set $X$ $\varepsilon$-\emph{separated} if it contains no $\varepsilon$-neighboring pair. Let $\mathcal{A}_\eps(X)$ and $\mathcal{N}_\eps(X)$ denote the set of antipodal pairs and the set of neighboring pairs, respectively. More precisely, we define:
\begin{align*}
     \mathcal{A}_\eps(X):=&\{\{x,y\}\subset X: \|x-y\|\geq 1-\eps\}, \\
     \vspace{4pt}\mathcal{N}_\eps(X):=&\{\{x,y\}\subset X: \|x-y\|\leq \eps\}. 
\end{align*}   
  \indent It is natural to expect that if we have many antipodal pairs, then some of the points are forced to be close to the boundary and the number of neighboring pairs increases. Furthermore, as we increase the number of antipodal pairs, we can also expect some clustering of the points. Hence, we are interested in how small the ratio $R_\eps(X):=|\mathcal{N}_\eps(X)|/|\mathcal{A}_\eps(X)|$  can be when the denominator is nonzero. This line of research was recently initiated by Steinerberger \cite{steinerberger2025antipodesimpliesneighbors}. He showed the following:

\begin{thm}[Steinerberger
]\label{thm:steinerberger}
There are constants $c_1,c_2>0$ such that for all $0<\varepsilon <0.5$ and any finite $X\subset \mathbb{R}^2$ of diameter at most $1$ with $|X|\ge c_1\varepsilon^{-2}$, we have $R_\eps(X)\ge c_2\sqrt[4]{\eps^3/\log(\eps^{-1})}$, i.e.
    \[
\big|\{i<j:\|x_i-x_j\|\le \eps\}\big|
\ge c_2\cdot \frac{\eps^{3/4}}{(\log\eps^{-1})^{1/4}}\cdot
\big|\{i<j:\|x_i-x_j\|\ge 1-\eps\}\big|.
\]

\end{thm}

Steinerberger provided a number of examples where the ratio is of order $\eps^{1/2}$, one of which is a regular polygon on sufficiently many vertices, and asked whether this threshold is optimal. Very recently, Korsky \cite{korsky2026optimalspectralboundsantipodal} showed that this is achievable up to a polylogarithmic factor. 

\begin{thm}[Korsky
]\label{thm:korsky}
There are constants $c_1,c_2>0$ such that for any $0<\varepsilon<0.5$  and  any finite $X\subset \mathbb{R}^2$ of diameter at most $1$ with $|X|\ge c_1\varepsilon^{-2}$, we have
$R_\eps(X)\ge c_2\sqrt{\eps/\log(\eps^{-1})}$, i.e.
    \[
\big|\{i<j:\|x_i-x_j\|\le \eps\}\big|
\ge c_2\cdot \frac{\eps^{1/2}}{(\log\eps^{-1})^{1/2}}\cdot
\big|\{i<j:\|x_i-x_j\|\ge 1-\eps\}\big|.
\]

\end{thm}

Theorems~\ref{thm:steinerberger} and~\ref{thm:korsky} are slightly restated here, as their authors use a different notion for pairs, which allows an element to form a pair with itself. Under their convention, the neighboring count is $2|\mathcal{N}_\eps(X)|+|X|$, while the antipodal count is $2|\mathcal{A}_\eps(X)|$. Proposition~\ref{prop:many-close-pairs} shows that the diagonal term can be absorbed into a constant factor when $|X|\geq c_1\eps^{-2}$, giving the formulations above. 
With our convention, a size condition of order $\eps^{-2}$ for $X$ is necessary: take an $\eps$-separated\linebreak grid in a disk of diameter $1$, add a diametral pair, and delete grid points within distance $\eps$ of its endpoints. This gives $\Theta(\eps^{-2})$ points with an antipodal pair but no neighboring pair. 

In this paper we confirm Steinerberger's conjecture. 

\begin{thm}\label{thm:main_single}
There are universal constants $c_1,c_2>0$ such that for any $0<\varepsilon<1$  and  any finite $X\subset \mathbb{R}^2$ of diameter at most $1$ with $|X|\ge c_1\varepsilon^{-2}$ we have $R_\eps(X)\ge c_2\eps^{1/2}$, i.e.
    \[
\big|\{i<j:\|x_i-x_j\|\le \eps\}\big|
\ge c_2\cdot \eps^{1/2}\cdot
\big|\{i<j:\|x_i-x_j\|\ge 1-\eps\}\big|.
\]
\end{thm}

Theorem \ref{thm:main_single} is a consequence of a more general result. We study the relationship between $|\mathcal{N}_{\eps_1}(X)|$ and $|\mathcal{A}_{\eps_2}(X)|$ for parameters $0<\eps_1,\eps_2<1$, focusing on the minimum possible value of their ratio. Our main result determines this quantity up to a constant factor.     

\begin{thm}\label{thm:main}
There are universal constants $c_1,c_2>0$ such that for any $0<\varepsilon_1,\varepsilon_2<1$ and any finite set $X\subset \mathbb{R}^2$ of diameter at most $1$ with $|X|\ge c_1\varepsilon_1^{-2}$ the following holds: 
\[
\dfrac{|\mathcal{N}_{\eps_1}(X)|}{|\mathcal{A}_{\eps_2}(X)|} \ge
\begin{cases}
c_2\cdot \eps_1^2\cdot \eps_2^{-3/2}, & \eps_1 \le \eps_2; \\
 c_2\cdot \eps_1^{1/2}, & \eps_1 > \eps_2.
\end{cases}
\]
\end{thm}

Theorem~\ref{thm:main}, together with the constructions in Section~\ref{sec:examples}, determines the optimal order of magnitude throughout the whole parameter range. In particular, it yields a tight resolution to Steinerberger's question, as it implies that $R_\eps(X)\geq c_2\eps^{1/2}$. Throughout the paper we focus on the two-parameter version. We will see examples showing that Theorem \ref{thm:main} is optimal, and we will discuss why a transition in behavior occurs at  $\eps_1=\eps_2$.

As far as the authors are aware, this is the first comprehensive result in this direction of study. There are a couple of results whose aim is to maximize the average distance between two points within a set of bounded diameter. For example, Witsenhausen \cite{Witsenhausen1974} studied the problem of maximizing $\sum_{i,j}\|x_i-x_j\|^2$ for a finite set $X:=\{x_1,x_2,\ldots,x_n\}\subset \mathbb{R}^2$ of diameter $1$, while Pillichshammer \cite{Pillichshammer2001, Pillichshammer2000O} analyzed $\sum_{i,j}\|x_i-x_j\|$. Theorem \ref{thm:main} relates to these results in the sense that having many pairs with maximum contribution (at least $1-\eps$) forces the existence of many pairs with a very small contribution (less than $\eps$). 

The proof of Theorem \ref{thm:main} was completed in close collaboration with the artificial intelligence ChatGPT 5.4.  
The spirit of our argument comes fairly close to the proofs of Theorem \ref{thm:steinerberger} and Theorem \ref{thm:korsky}. Following the ideas of Steinerberger, we begin by observing that only the points in $X$ that are within $\eps_2$ distance from the boundary of $\text{conv}(X)$ can contribute to $\eps_2$-antipodal pairs. Then we simplify the problem by covering the region near the boundary of $\text{conv}(X)$ with boxes of diameter less than $\eps_1$. When $\eps_1\leq \eps_2$, a natural choice for these boxes is given by squares of side length $\eps_1/2$. On the other hand when $\eps_1>\eps_2$, and especially when $\eps_2=o(\eps_1)$, these squares would cover a much larger area than needed, and a better choice is given by rectangles of dimension $\eps_1/2\times \eps_2/2$. 
This also explains the dichotomy in the conclusion of Theorem \ref{thm:main}. Next, we define a graph $G_X$, called the \emph{box graph}, on these boxes by connecting pairs that are at least $1-\eps_2$ far from each other. The problem reduces to bounding the largest eigenvalue of the adjacency matrix of the box graph. Using a key geometric lemma, we show that $G_X$ satisfies a natural sparsity condition. Roughly speaking, there is a universal constant $c>0$ such  that any vertex has at most $c\eps_2^{3\ }\eps_1^{-4}/t$ neighbors of degree at least $t$ for any integer $t>0$. This property allows us to derive the optimal bound on the eigenvalue of the box graph.  

After introducing the necessary metric and linear algebraic tools in Section \ref{sec:linalg}, we perform the reduction to the box graph in Section \ref{sec:reduction}. In Section \ref{sec:geom} we present Lemma \ref{quadrilateral}, which encompasses the geometry of the box graph and shows how it implies the sparsity condition (Lemma \ref{lem:tails}). This allows us to finish the proof of Theorem \ref{thm:main}. Next, in Section \ref{sec:examples} we exhibit various extremal examples which demonstrate the sharpness of our bounds. Finally, in Section \ref{sec:conclude}, we discuss some open problems and provide additional insight into the selection of test vectors we have investigated for the Collatz-Wielandt bound before reaching the effective choice from Proposition \ref{prop:sqrt-control}.

\section{Linear algebra and metric prerequisites}\label{sec:linalg}

In this short section, we recall a couple of useful results from linear algebra. The terminology we use is quite standard and our presentation is largely self-contained; however, the reader can consult \cite{booklinalg} for a more detailed discussion.


The results presented below could be generalized to a broader context, but for our purposes, it is enough to restrict ourselves to real-valued symmetric matrices. What is essential in this setting is that the eigenvalues of such matrices are real-valued, and the largest, denoted by $\lambda_1(A)$ for a matrix $A$, satisfies $\lambda_1(A)=\underset{\mathbf{x}\in \mathbb{R}^n\setminus\{0\}}{\max}\dfrac{\langle \mathbf{x},A\mathbf{x} \rangle}{\langle \mathbf{x},\mathbf{x} \rangle}$. 

When the symmetric matrix $A$ also has only non-negative entries, by the Perron-Frobenius theorem there is an eigenvector $\mathbf{x}\in \mathbb{R}^n$ for $\lambda_1(A)$ whose entries are all non-negative. This observation is crucial to proving the following useful proposition. 

\begin{prop}[Collatz--Wielandt bound]\label{vector-test-criterion}
 Let $A$ be a real symmetric $n\times n$ non-zero matrix with non-negative entries. Suppose there is a vector $\mathbf{x}\in \mathbb{R}^n_{> 0}$ and $\lambda>0$ such that $(A\mathbf{x})_i\leq \lambda x_i$ for each coordinate $i\in [n]$. Then $$\lambda\geq \lambda_1(A)>0.$$
\end{prop}

\begin{proof}
  Recall that $\lambda_1(A)=\max_{||v||=1}\langle v,Av\rangle$, so clearly $\lambda_1(A)>0$. Moving forward, pick an eigenvector $v$ for $\lambda_1(A)$ with non-negative coordinates $v_i$. Set $t:=\min \{x_i/v_i: v_i\neq 0\}$ and note that for each coordinate $i$ we have $\lambda x_i\geq (A\mathbf{x})_i\geq \left(A(tv)\right)_i$ since $x_j\geq tv_j$ for all $j\in[n]$ and $A$ has nonnegative entries. It follows that $\lambda x_i\geq t(Av)_i=t\lambda_1(A)v_i$. Finally, there is a coordinate $j$ for which $x_j=tv_j\neq 0$ and plugging this into the inequality above yields $\lambda\geq \lambda_1(A)$. 
\end{proof}

From this point onward, the eigenvalues, eigenvectors, and spectrum of a graph will refer to those of its adjacency matrix, which is the square $(0,1)$-matrix that encodes the edges of the graph.

We conclude this section with a metric result which roughly says that if a set $X$ of bounded diameter is large enough, then the number of pairs of points from $X$ which are close to each other grows at least linearly with respect to $|X|$. We shall use this simple packing estimate to rule out the case where too few points lie in the boundary layer.

\begin{prop}\label{prop:many-close-pairs}
    For any real number $a\geq 1$ there is a constant $c_a>0$ such that for any $0<\eps<1$ and any finite set $X\subset \mathbb{R}^2$ of diameter at most $1$ with size $|X|\geq c_a\cdot \varepsilon^{-2}$, we have: 
    $$|\mathcal{N}_\eps(X)|\geq a|X|.$$
\end{prop}

\begin{proof}
    Since $X$ has diameter at most $1$, its convex hull $\text{conv}(X)$ lies in a square of side length $1$. A square grid therefore shows that $\text{conv}(X)$ can be covered by a collection of $M:=\lfloor r\cdot\eps^{-2}\rfloor$ balls of radius $\eps/2$, where $r>0$ is a universal constant. We distribute the points of $X$ to these balls such that each point gets assigned to exactly one of the balls containing it. Suppose the $M$ balls contain $y_1,y_2,\ldots,y_M$ points from $X$ after the assignment, with $\sum_{i=1}^M y_i=|X|$. Since any pair of points assigned to the same ball is $\eps$-neighboring, we infer that 
\[
{\left|\mathcal{N}_\eps(X)\right|\geq \sum\limits_{i\in [M]}\dbinom{y_i}{2}= \dfrac{1}{2}\Bigl(\Bigl(\sum\limits_{i\in [M]}y_i^2\Bigr) - |X|\Bigr)}
.
\]
  \indent  By the Cauchy-Schwarz inequality we get $\sum_{i=1}^M y_i^2\geq |X|^2/M$. Hence, for any $c_a\geq 2ra+r$, we obtain, as desired, that 
    $$\left|\mathcal{N}_\eps(X)\right|\geq\frac{|X|^2}{2M}-\dfrac{|X|}{2}\geq \frac{c_a\cdot\eps^{-2}\cdot |X|}{2r\cdot \eps^{-2}}-\dfrac{|X|}{2}\geq \frac{(c_a-r)\cdot|X|}{2r}\geq a|X|.$$
\end{proof}

\section{The box graph and the spectral reduction}\label{sec:reduction}

    In this section we define a natural auxiliary graph for a point set $X$ of finite diameter, whose vertices will represent clusters of neighboring points, while its edges will encode antipodal pairs. We then show how to reduce the problem to bounding the largest eigenvalue of this graph. 
    
    Let us start by observing that it is enough to prove Theorem \ref{thm:main} only for $\eps_1\leq \eps_2$. Indeed, if $\eps_1>\eps_2$, then any $\eps_2$-antipodal pair is $\eps_1$-antipodal as well. Hence, using the $\eps_2=\eps_1$ case, we obtain $\left|\mathcal{N}_{\eps_1}(X)\right|\geq  c_2\cdot \eps_1^{2}\cdot\eps_1^{-3/2}\left|\mathcal{A}_{\eps_1}(X)\right| = c_2\cdot \eps_1^{1/2}\left|\mathcal{A}_{\eps_1}(X)\right|\geq c_2\cdot \eps_1^{1/2}\left|\mathcal{A}_{\eps_2}(X)\right|$. Therefore, from now on we assume that $\eps_1\le\eps_2$.

We may also assume that $\eps_2\leq 0.01$. Indeed, the grid estimate in the proof of Proposition~\ref{prop:many-close-pairs} gives $|\mathcal{N}_{\eps_1}(X)|\geq c\cdot \eps_1^2|X|^2$ when $|X|\geq C\eps_1^{-2}$. For $\eps_2\geq 0.01$, this and $|\mathcal{A}_{\eps_2}(X)|\leq |X|^2/2$ imply the desired bound, eventually with a smaller constant in front. This also ensures that no antipodal pair lies in a single box.



Let $K:=\text{conv}(X)$ denote the convex hull of $X$. If a point $x\in X$ is at a distance greater than $\varepsilon_2$ from the boundary of $K$, i.e. if $\text{dist}(x,\partial K)>\varepsilon_2$, then $x$ does not contribute  any pair to $\mathcal{A}_{\eps_2}(X)$. Therefore, if we define $K_{\eps_2}:=\{x\in K:\text{dist}(x,\partial K)\le {\eps_2}\}$ and set $X_{\eps_2}=X\cap K_{\eps_2}$, we get $|\mathcal{A}_{\eps_2}(X_{\eps_2})|=|\mathcal{A}_{\eps_2}(X)|$, and $|\mathcal{N}_{\eps_1}(X_{\eps_2})|\leq |\mathcal{N}_{\eps_1}(X)|$. Thus, we focus on covering $X_{\eps_2}$ while retaining the size hypothesis on the original set $X$.

Let $S$ be a maximal $\eps_1/4$-separated set of points inside $K_{\eps_2}$. The squares of side length $\eps_1/2$ centered at the points in $S$ cover $K_{\eps_2}$, otherwise an uncovered point could be added to $S$. Their number $N:=|S|$ satisfies $N=O(\varepsilon_2\varepsilon_1^{-2})$, since the disjoint disks of radius $\eps_1/8$ whose centers are the points of $S$ lie in the $(\eps_2+\eps_1/8)$-neighborhood of $\partial K$. The area this neighborhood is $O(\eps_2)$ since the perimeter of $K$ is bounded by a constant. From now on, we will refer to these squares as \emph{boxes} $B_1,B_2,\ldots,B_N$.

The \emph{box graph} $G_X$ is the graph on vertex set $[N]:=\{1,2,3,\ldots,N\}$ where we connect $i$ and $j$ if there are $x\in X\cap B_i$ and $y\in X\cap B_j$ such that $\|x-y\|\geq 1-\eps_2$. We write $i\sim j$ to express that $i$ and $j$ are adjacent vertices in $G_X$ and also let $d_i:=\text{deg}_{G_X}(i)$. 
Furthermore, we assign each vertex of $X_{\eps_2}$ to one of the boxes containing it, let $n_i$ denote the number of points of $X$ that are assigned to the box $B_i$.

We note that the box graph is not unique for a set $X$, as there are different ways to cover $K_{\eps_2}$ by boxes of size $\eps_1/2$, but any one will be suitable for our purposes. We will use the notation $G_X(\eps_1,\eps_2)$ to make it clear that we refer to a box graph where we used squares of side length $\eps_1/2$ to cover $K_{\eps_2}$.  We will now see how $G_X(\eps_1,\eps_2)$ helps us compare $\mathcal{N}_{\eps_1}(X)$ with $\mathcal{A}_{\eps_2}(X)$.  Note that if $|\mathcal{N}_{\eps_1}(X)|\ge (\eps_1/\eps_2)^2 \cdot|\mathcal{A}_{\eps_2}(X)|$ then Theorem \ref{thm:main} holds, so it is enough to consider the case $|\mathcal{N}_{\eps_1}(X)|< (\eps_1/\eps_2)^2 \cdot|\mathcal{A}_{\eps_2}(X)|$.

\begin{prop}\label{prop:eigen}
    There is a universal constant $c_1>0$ such that for any $0<\eps_1\leq \eps_2\leq 0.01$ the following statement holds. Suppose $X\subset \mathbb{R}^2$ is a finite set of diameter at most $1$, whose size is $|X|\geq c_1\varepsilon_1^{-2}$, for which $|\mathcal{N}_{\eps_1}(X)|< (\eps_1/\eps_2)^2 \cdot|\mathcal{A}_{\eps_2}(X)|$. Then for any box graph $G_X(\eps_1,\eps_2)$ associated to it, the following inequality holds: $$2\lambda_1\big(G_X(\eps_1,\eps_2)\big)\cdot|\mathcal{N}_{\eps_1}(X)|\geq |\mathcal{A}_{\eps_2}(X)|.$$
\end{prop}

\begin{proof}
 Suppose $G_X(\eps_1,\eps_2)$ is an $N$-vertex graph and let $B$ denote its adjacency matrix. Furthermore, let $\mathbf{n}:=(n_1,n_2,\ldots,n_N)$ be the vector whose $i^{\text{th}}$ coordinate counts how many points of $X_{\eps_2}$ are assigned to the box $B_i$.
    
On the one hand, any pair $\{x,y\}\subset X$ with $\|x-y\|\geq 1-\eps_2$ must come from two boxes $B_i$ and $B_j$ which are neighbors in the box graph. This implies that \[
|\mathcal{A}_{\eps_2}(X)|\leq \sum_{\{i,j\}\in E(G_X)}n_in_j
=\dfrac{\langle\mathbf n,B\mathbf n\rangle}{2}.\] 
  If we show that $4|\mathcal{N}_{\eps_1}(X)|\geq \sum_{i=1}^Nn_i^2=\langle\bf{n},\bf{n}\rangle$, then we can conclude that:
    $$\dfrac{|\mathcal{A}_{\eps_2}(X)|}{|\mathcal{N}_{\eps_1}(X)|}\leq \dfrac{4\langle\bf{n},\emph{B}\bf{n}\rangle}{2\langle\bf{n},\bf{n}\rangle}\leq 2\cdot \max_{\bf{x}\in\mathbb{R}^N\setminus\{0\}}\dfrac{\langle\bf{x},\emph{B}\bf{x}\rangle}{\langle\bf{x},\bf{x}\rangle}=2\lambda_1(B).$$
\indent Hence, we focus on lower bounding $|\mathcal{N}_{\eps_1}(X)|$. Any pair of points that lie in the same box $B_i$ must be $\eps_1$-neighboring, thus  
\[|\mathcal{N}_{\eps_1}(X)|\geq\sum_{i\in[N]}\binom{n_i}{2}
=\frac12\sum_{i\in[N]}n_i^2-\frac{|X_{\eps_2}|}{2}.\] 
\indent We are left to show that we can pick $c_1$ so that $|X_{\eps_2}|/2$ is smaller than $\tfrac14\sum_{i\in[N]}n_i^2$.
Indeed, $\sum_{i\in[N]} n_i^2\geq N^{-1}|X_{\eps_2}|^2$ by the Cauchy-Schwarz inequality. Hence, whenever $|X_{\eps_2}|\geq 4N$, we obtain $
\sum_{i\in[N]}n_i^2\geq\frac{|X_{\eps_2}|^2}{N}\geq4|X_{\eps_2}|.$
In this case, it follows that $|\mathcal{N}_{\eps_1}(X)|\geq \frac{1}{4}\sum_{i\in [N]}n_i^2$.

As $N=O\left(\varepsilon_2\cdot\varepsilon_1^{-2}\right)$, there is a universal constant $c_3>4\eps_1^2\eps_2^{-1}N$ for any box graph. Using Proposition \ref{prop:many-close-pairs} with $a=c_3^2/2$, we can fix a  constant $c_1\ge1$ such that  $|\mathcal{N}_{\eps_1}(X)|\geq (c_3^2/2)\cdot|X|$ whenever $|X|\geq c_1\cdot \eps_1^{-2}$. We claim that $|X_{\eps_2}|\geq 4N$ for this choice of $c_1$.

 Indeed, suppose $|X_{\eps_2}|<4N$. Then we get
$\left|\mathcal{A}_{\eps_2}(X)\right|\leq |X_{\eps_2}|^2/2<8N^2\le c_3^{2\ }\eps_1^{-4\ }\eps_2^{2}/2$. Also, using $|X|>c_1\eps_1^{-2}$ and $c_1>1$, we have  $|\mathcal{N}_{\eps_1}(X)|\geq (c_3^2/2)\cdot|X|\ge(c_3^2/2)\cdot\eps_1^{-2}$.
 This would give us $|\mathcal{N}_{\eps_1}(X)|\geq (\eps_1/\eps_2)^2\cdot|\mathcal{A}_{\eps_2}(X)|$, contradicting the assumption of the proposition. 
\end{proof}
 

Therefore, our current focus is on finding an optimal upper bound for $\lambda_1\big(G_X(\eps_1,\eps_2)\big)$.
We conclude this section by observing the trivial bound given by $\lambda_1(G_X)\leq |V(G_X)|=O\big(\varepsilon_2 \varepsilon_1^{-2}\big)$. Indeed, this follows from Proposition \ref{vector-test-criterion} applied with the vector $\mathbf{1}:=(1,1,\ldots,1)$.

\section{Geometry gets involved}\label{sec:geom}
We are yet to exploit the geometry of the box graph. For convenience, if $P_1$ and $P_2$ are two points in the plane we let $|P_1P_2|$ denote the length of the closed segment determined by them.

 We start with a simple lemma on convex quadrilaterals. The lemma itself is very similar in nature to Theorem \ref{thm:main}, it says that if we have two pairs of points at distance $1-\varepsilon$, then their endpoints cluster to some degree. 

\begin{lem}\label{quadrilateral}
    Let $0<\varepsilon<1$ and let $A,B,C,D$ be the vertices of a convex quadrilateral in counterclockwise order in the plane whose diameter is at most $1$. If the lengths of the segments $AD$ and $BC$ are $\geq 1-\varepsilon$, then $|AB|\cdot |CD|<3\varepsilon$. 
\end{lem}

\begin{proof}
     If either $|AB|< 3\eps$ or $|CD|< 3\eps$ then we are done. First, we claim that we can assume that $|AC|=|BD|=1$. Indeed, consider what happens if we rotate the point $C$ clockwise around $B$. Since $ABCD$ is convex, both $AC$ and $CD$ increase, and the quadrilateral remains convex until $C$ reaches the line $AB$, see Figure \ref{fig:rotate}. As $|AB|\ge 3\eps$, $|AB|+|BC|>1$ i.e., the length of $AC$ reaches $1$ before $C$ reaches the $AB$ line. Hence, we can assume that $|AC|=1$. Similarly, we can rotate $D$ counterclockwise around $A$ until $|DB|=1$. (The rotation might increase the length of $CD$ above 1, but this will not cause any problems in the following argument.)
   
     \begin{figure}[!ht]
         \centering
         \includegraphics[width=0.3\linewidth]{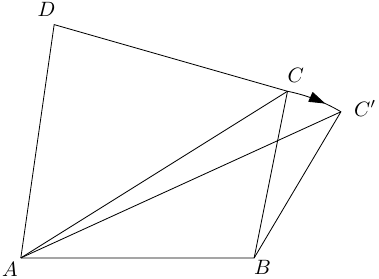}
         \caption{We can rotate $C$ around $B$ until $AC=1$, it only increases $|AC|$ and $|CD|$.}
         \label{fig:rotate}
     \end{figure}

     Suppose then $|AB|=2h$ with $3\varepsilon\leq 2h\leq 1$ and set a coordinate system in the plane such that $A=(-h,0)$ and $B=(h,0)$. Then $C$ and $D$ must lie on two circular arcs (see Figure \ref{fig:quadlemma}) and let $C=(c_x,c_y)$ and $D=(d_x,d_y)$ be their coordinates with $c_x,c_y,d_y\ge0$ and $d_x\le0$. We need to show that $|CD|<3\eps/2h$, or equivalently, that $(c_x-d_x)^2+(c_y-d_y)^2<9\eps^2\cdot(2h)^{-2}$. 

    \begin{figure}[!ht]
        \centering
        \includegraphics[width=0.4\linewidth]{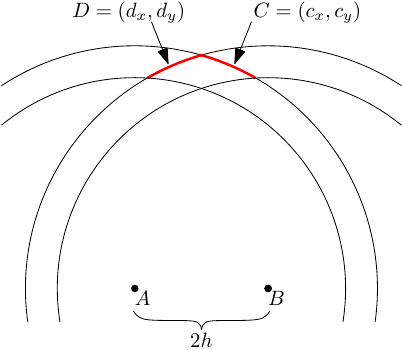}
        \caption{
        $C$ and $D$ must lie on the two small red arcs, respectively.}
        \label{fig:quadlemma}
    \end{figure}

   Since $|BC|\geq 1-\eps$ and $|AC|=1$, this translates to $(c_x-h)^2+c_y^2\geq (1-\eps)^2>1-2\eps$ and $(c_x+h)^2+c_y^2=1$, respectively. By subtracting the second inequality from the first we obtain $c_x\leq \eps/2h\leq 1/3$.  Analogously we get $-d_x\leq \eps/2h\leq 1/3$ by considering the point $D$. 

   To bound the $y$-axis coordinates $c_y$ and $d_y$ note that we can write $c_y=f(c_x)$ and $d_y=f(-d_x)$ for the map $f:[0,1/3]\to (0,1)$ given by $f(t):=\sqrt{1-(t+h)^2}$. By the mean value theorem,
   there is $z\in [0,1/3]$ such that $|c_y-d_y|=|c_x-(-d_x)|\cdot |f^\prime(z)|$, where $f^\prime:[0,1/3]\to(-\infty,0)$ is the derivative of $f$, given by $f^\prime(t):=-(t+h)/ \sqrt{1-(t+h)^2}$. It is not hard to see that the function $-f^\prime$ is increasing on $[0,1/3]$, hence $|f^\prime(z)|\leq |f^\prime(1/3)|\leq 5/\sqrt{11}<\sqrt{5}$. It follows that $(c_y-d_y)^2< 5(c_x+d_x)^2$.
    Using that $0<c_x\leq \eps/2h$ and $0<-d_x\leq \eps/2h$ we obtain $(c_x-d_x)^2<\eps^2/h^2$ and $(c_y-d_y)^2\leq 5(c_x+d_x)^2\le 5\eps^2/4h^2$. 
    
    It follows that $|CD|^2\le9\eps^2/4h^2$, which implies $|CD|\le 3\eps/2h=3\eps/|AB|$.
\end{proof}

We note that the sharp bound in Lemma \ref{quadrilateral} seems to be $|AB|\cdot|CD|\le 2\varepsilon-\eps^2$, however, any estimate of order $O(\eps)$ suffices for our purposes; therefore, we provided a weaker bound with a simpler proof.

Our next lemma  helps us identify, from a geometric perspective, the center of a box with its corresponding vertex in the box graph.   

\begin{lem}\label{box-point}
    Let $G:=G_X(\varepsilon_1,\varepsilon_2)$ be an $N$-vertex box graph associated to a finite set $X\subset \mathbb{R}^2$ of diameter at most $1$ and let $B_i$ and $B_j$ be the two boxes corresponding to vertices $i\sim j$ in $G$. Then $1-\eps_1-\eps_2< |P_iP_j|\leq 1$, where $P_i$ and $P_j$ are the centers of the boxes $B_i$ and $B_j$, respectively. 
\end{lem}

\begin{proof}
    Suppose the points $X_i$ and $X_j$ in $X$ certify that $i\sim j$, i.e. that $|X_iX_j|\geq 1-\eps_2$. We also know that $|P_iX_i|<\eps_1/2$ and $|P_jX_j|<\eps_1/2$, therefore by the triangle inequality we obtain that $|P_iP_j|>|X_iX_j|-|X_iP_i|-|X_jP_j|>1-\eps_2-\eps_1/2-\eps_1/2=1-\eps_1-\eps_2$. Furthermore, $P_i$ and $P_j$ lie inside $\text{conv}(X)$, hence $|P_iP_j|\leq 1$, finishing the proof.
\end{proof}

We are now ready to derive the key tail estimate on the local edge density of the box graph around a specific vertex. 

\begin{lem}\label{lem:tails}
There is a universal constant $c_0>0$ for which the following statement holds. 
Let $0<\eps_1\leq \eps_2\leq 0.01$ and let $G:=G_X(\eps_1,\eps_2)$ be an $N$-vertex box graph associated to a finite set $X\subset \mathbb{R}^2$ of diameter at most $1$. Then, for every vertex $k\in [N]$ and integer $t\ge 1$, we have: 
$$|\{j\sim k: d_j\ge t\}|\le c_0\cdot \eps_2^3\cdot \eps_1^{-4}/t.$$
\end{lem}

\begin{proof}
Let $N_t=\{j\sim k: d_j\ge t\}$. First note that if $c_0$ is large enough, then the statement holds for $t\le 5$ and also for $|N_t|\le 5$. Hence, we may assume that $t>5$ and $|N_t|>5$. We show that in these cases any $c_0\geq 2^{23}$ works. Suppose by contradiction that $|N_t|> c_0\cdot \eps_2^3\cdot \eps_1^{-4}/t$.

Recall that we have labeled the boxes by $B_1,\dots, B_N$ and let $P_1,\dots, P_N$ denote the centers of the boxes, respectively. By Lemma~\ref{box-point}, all centers $P_l$ of boxes $B_l$ belonging to the set $N_t$ are at distance between $1-2\eps_2$ and $1$ from $P_k$. Consequently, the directions of the rays $P_kP_l$ with $l\in N_t$ lie in an angular interval of length at most $\pi/2$: indeed, the cosine rule gives positive cosine for any angle of the form $P_iP_kP_j$ with $i,j\in N_t$. Let $<_k$ be their counterclockwise order in this interval, breaking ties arbitrarily.

\begin{figure}[!ht]
    \centering
    \includegraphics[width=0.5\linewidth]{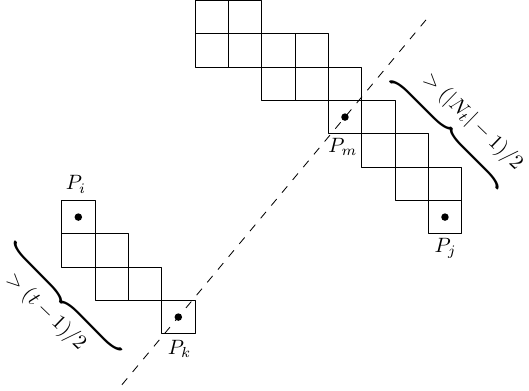}
    \caption{Finding four boxes, $B_k,B_m,B_i,B_j$, that are far from each other.}
    \label{fig:tail}
\end{figure}

 Pick a middle box from $N_t$ according to $<_k$ and label it $B_m$. It has at least $t$ neighbors. Among its neighbors other than $k$, at least $(t - 1)/2$ correspond to boxes whose centers lie in one of the
two half-planes determined by the line $P_kP_m$. Fix such a half-plane
and let $B_i$ be a box among these for which the angle $\angle P_iP_mP_k$ is maximized, see Figure \ref{fig:tail}. By the choice of box $B_m$, at least $\lfloor (|N_t|-1)/2\rfloor$ elements of $N_t\setminus\{m\}$ have their centers in the opposite half-plane. Let $B_j$ be one among these for which the angle $\angle P_jP_kP_m$ is maximized.

We claim, on one hand, that $|P_iP_k|\cdot |P_jP_m|\leq 6\varepsilon_2$. 
Since $i\sim m$, $j\sim k$, $m\sim k$, and $\varepsilon_1\leq\varepsilon_2$, by Lemma \ref{box-point} we get $|P_iP_m|,\ |P_kP_j|,\ |P_kP_m|>1-2\varepsilon_2$.
Now let $Q$ be the intersection of the segment $P_iP_j$ with the line $P_kP_m$. If $Q$ lies on the segment $P_kP_m$, then the points $P_i,P_k,P_j,P_m$ form a convex quadrilateral and Lemma \ref{quadrilateral}, applied with parameter $2\varepsilon_2$, gives the claim.

Otherwise, suppose that $Q$ lies on the ray from $P_k$ opposite $P_m$. We have $|P_mQ|\leq 1$ as $Q$ and $P_m$ lie in $\text{conv}(X)$. This implies that
$|P_kQ|<2\varepsilon_2$ and then $|QP_j|>|P_kP_j|-|P_kQ|>1-4\varepsilon_2$.
As $|P_iP_j|\leq1$, this also gives $|P_iQ|<4\varepsilon_2$. By triangle inequality, we obtain $|P_iP_k|<6\varepsilon_2$ and therefore the claim, since $|P_jP_m|\leq1$. The case when $Q$ lies on the other side of the segment $P_kP_m$ follows analogously, as then a symmetric argument gives $|P_jP_m|<6\varepsilon_2$. If either $P_i$ or $P_j$ lies on the line $P_kP_m$, then respectively $|P_iP_k|<2\varepsilon_2$ or $|P_jP_m|<2\varepsilon_2$, and the claimed product upper bound follows immediately.

On the other hand, using the assumption and the box graph degree properties, we claim that $|P_iP_k|\cdot |P_jP_m|\ge 8 \eps_2$, which represents an obvious contradiction. We show this by lower bounding each of  
 $|P_iP_k|$ and $|P_jP_m|$ separately.

Consider first $|P_iP_k|$. Let $\alpha$ denote the measure of $\angle P_iP_mP_k$ in radians. By the choice of $P_i$,
at least $(t-1)/2\geq t/3$ centers of neighbors of $m$ lie in the annular sector $S$ centered at $P_m$, with radii $1-2\eps_2$ and $1$, and bounded by the rays $P_mP_k$ and $P_mP_i$. Recall that the centers $P_l$ are $\eps_1/4$-separated, hence if we partition the radial interval $[1-2\eps_2,1]$ and the angular interval $[0,\alpha]$ into intervals of length $\eps_1/(4\sqrt{2})$, then each resulting polar cell contains at most one box center. Indeed, if two points $(r,\varphi)$ and $(s,\psi)$ lie in the same cell, then
\[\big|re^{i\varphi}-se^{i\psi}\big|^2
=(r-s)^2+2rs(1-\cos(\varphi-\psi))
\leq (r-s)^2+(\varphi-\psi)^2
<\frac{\eps_1^2}{16}.\]
\indent Consequently, $\ \dfrac{t}{3}\leq
\left(1+8\sqrt{2}\cdot \dfrac{\eps_2}{\eps_1}\right)
\left(1+4\sqrt{2}\cdot \dfrac{\alpha}{\eps_1}\right).$
We next relate $\alpha$ to $|P_iP_k|$:
$$|P_iP_k|^2=\left(|P_mP_i|-|P_mP_k|\right)^2+4|P_mP_i|\cdot|P_mP_k|\sin^2(\alpha/2)\geq 4(1-2\varepsilon_2)^2\sin^2(\alpha/2)\geq \alpha^2/4.$$
\indent Moreover, $|P_iP_k|\geq\eps_1/4$, so we get $1+4\sqrt{2}\varepsilon_1^{-1}\alpha\leq 1+8\sqrt{2}\eps_1^{-1}|P_iP_k|
\leq
16\eps_1^{-1}|P_iP_k|$, while also $1+8\sqrt{2}\eps_1^{-1}\eps_2\leq 13\eps_1^{-1}\eps_2$.
Plugging these estimates above, we obtain $t\leq 39\cdot 16\eps_2\eps_1^{-2}|P_iP_k|$, which implies that $|P_iP_k|\geq
\eps_1^2\eps_2^{-1}\cdot t\cdot 2^{-10}$.

By the choice of $P_j$, at least $\lfloor(|N_t|-1)/2\rfloor\geq |N_t|/3$ centers of boxes from $N_t$ lie in the annular sector $S$ centered at $P_K$, with radii $1-2\eps_2$ and $1$, and bounded by the rays $P_kP_m$ and $P_kP_j$. The same argument, with $P_k$ in place of $P_m$, gives $|P_jP_m|
\geq\eps_1^2\eps_2^{-1}\cdot |N_t|\cdot 2^{-10}$.

By combining these two inequalities, we obtain
$|P_iP_k|\cdot |P_jP_m|
\geq
\eps_1^4\eps_2^{-2}\cdot t\cdot|N_t|\cdot2^{-20}$. 
Using the assumption $|N_t|> c_0\cdot \eps_2^3\cdot \eps_1^{-4}/t$ and that $c_0\geq2^{23}$, we arrive at the desired contradiction that 
$|P_iP_k|\cdot |P_jP_m|\geq 8\eps_2$.
We conclude that our assumption is false and so $|N_t|\leq  c_0\cdot \eps_2^3\cdot \eps_1^{-4}/t$.
\end{proof}

With these tools in hand, we are now ready to bound the largest eigenvalue of the box graph.

\begin{prop}\label{prop:sqrt-control}
    There exists a universal constant $c_0>0$ such that for any  $0<\eps_1\leq \eps_2\leq 0.01$ and for any finite set of points $X\subset \mathbb{R}^2$ of diameter at most $1$, and any box graph $G:=G_X(\eps_1,\eps_2)$ associated to it, the following spectral bound holds:$$\lambda_1\big(G_X(\eps_1,\eps_2)\big)\leq 2\sqrt{c_0}\,\eps_2^{3/2}\eps_1^{-2}.$$ 
\end{prop}

\begin{proof}
  Let $c_0$ be the constant given by Lemma \ref{lem:tails} and put $M:=\eps_2\eps_1^{-2}$. We can remove isolated vertices without changing its largest eigenvalue and also assume $G$ is nonempty.  So
   let $V(G)=[N]$, $d_i=\text{deg}_{G_X}(i)$ for all $i\in[N]$ and let $\mathbf{r}:=(\sqrt{d_1},\sqrt{d_2},\ldots,\sqrt{d_N} )$. We claim that $(B\mathbf{r})_i\leq 2\eps_2\eps_1^{-1}\sqrt{c_0M\cdot d_i}$ for each $i\in[N]$, where $B$ is the adjacency matrix of $G$. This claim implies the desired result by Proposition \ref{vector-test-criterion}, finishing the proof. 

To proceed, note that
$(B\mathbf{r})_i=\sum_{j\sim i}\sqrt{d_j}$.
For any real number $s>0$, applying Lemma \ref{lem:tails} with the integer
$t:=\lfloor s^2\rfloor+1$ for the set $A_s:=\{j\sim i:d_j>s^2\}$ gives
\[  \bigl|A_s\bigr|
    =\bigl|\{j\sim i:d_j\ge\lfloor s^2\rfloor+1\}\bigr|
    \le\frac{c_0M\eps_2^2\eps_1^{-2}}{\lfloor s^2\rfloor+1}
    \le\frac{c_0M\eps_2^2\eps_1^{-2}}{s^2}.
\]
\indent Moreover, $A_s$ has size at most $d_i$. Hence, setting
$a:=\eps_2\eps_1^{-1}\sqrt{c_0M/d_i}>0$, we obtain:
\[
    \sum_{j\sim i}\sqrt{d_j}
    =\int_0^\infty\bigl|A_s\bigr|\,ds
    \le\int_0^a d_i\,ds
       +\int_a^\infty\frac{c_0M\eps_2^2\eps_1^{-2}}{s^2}\,ds
    =d_i a+\frac{c_0M\eps_2^2\eps_1^{-2}}{a}
     =2\eps_2\eps_1^{-1}\sqrt{c_0M d_i},
\]
which proves the claim.\end{proof}

Putting it all together, we obtain a proof of Theorem \ref{thm:main}.
\begin{proof}[Proof of Thm \ref{thm:main}]
    As previously discussed, it is enough to prove the case $0<\eps_1\leq\eps_2\leq0.01$. 
    
    If ${|\mathcal{N}_{\eps_1}(X)|/|\mathcal{A}_{\eps_2}(X)|\geq (\eps_1/\eps_2)^2}$ then we are already done since  $(\eps_1/\eps_2)^2=\Omega(\eps_1^2\cdot \eps_2^{-3/2})$. On the other hand, if $|\mathcal{N}_{\eps_1}(X)|/|\mathcal{A}_{\eps_2}(X)|< (\eps_1/\eps_2)^2$, then by using Proposition \ref{prop:eigen} we get an absolute constant $c_0>0$ such that  $|\mathcal{N}_{\eps_1}(X)|/|\mathcal{A}_{\eps_2}(X)|\geq \left(2\lambda_1(G)\right)^{-1}$ for any $|X|\ge c_1\eps_1^{-2}$ and any box graph $G=G_X(\eps_1,\eps_2)$. Finally, Proposition \ref{prop:sqrt-control} helps us draw the conclusion that 
    $$\dfrac{|\mathcal{N}_{\eps_1}(X)|}{|\mathcal{A}_{\eps_2}(X)|}=\Omega\big(\lambda_1(G)^{-1}\big) =\Omega\left(\eps_1\eps_2^{-1}\sqrt{\eps_1^2/\eps_2}\right)= \Omega\left( \eps_1^2\cdot \eps_2^{-3/2}\right).$$ 
    \end{proof}


\section{Constructions}\label{sec:examples}

In this section, we provide several examples that show that the bounds from Theorem \ref{thm:main} are tight up to a constant factor. 

\subsection{Annulus construction}
 A natural idea is to consider a set of points which are uniformly distributed near the boundary of a disk of radius $1/2$. Our first construction follows this idea, as $X$ will be a subset of an annulus with inner and outer radii $(1-\eps_2)/2$ and $1/2$, respectively. We simply place the points as concentric regular polygons, as in Figure \ref{fig:annulus}.

For convenience, given two integers $a\leq b$, we will denote by $\mathbb{Z}[a,b]$ the set $\{a,a+1,\ldots,b\}$, i.e. the set of integers ranging from $a$ to $b$. We will also write $\mathbb{Z}_n$ for the set of remainders modulo $n$, i.e., $\mathbb{Z}_n:=\mathbb{Z}[0,n-1]$ for any positive integer $n$.

\begin{exam}[Annulus construction]
    Let $0<\delta<1$ and let $r,s\geq 2$ be integers. We define:
    $$AU(\delta,r,s):=\left\{ \left(\dfrac{1}{2}-\dfrac{\delta l}{2s}\right)\left( \cos\left(\frac{\pi k}{r}\right), \sin\left(\frac{\pi k}{r}\right)\right) : (k,l)\in \mathbb{Z}[-r,r-1]\times \mathbb{Z}_s\right\}.$$
\end{exam}

\begin{figure}[!ht]
    \centering
    \includegraphics[width=0.3\linewidth]{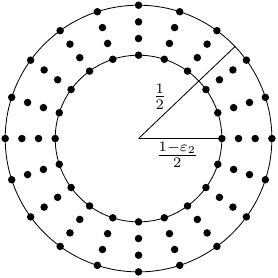}
    \caption{The annulus construction for $r=20$, $s=4$.}
    \label{fig:annulus}
\end{figure}

Our goal is to show the following estimate.

\begin{prop}\label{ref:annularconstr}
    Let $0<2\eps\leq\delta\leq 1/2$ and let $r,s\geq \eps^{-2}$ be integers. Then:
    \begin{align*}
     \left|\mathcal{A}_{\delta}\big(AU(\delta,r,s)\big)\right|=&\ \Theta\left(\sqrt{\delta}\cdot r^2s^2\right); \\
     \left|\mathcal{N}_{\eps}\big(AU(\delta,r,s)\big)\right|=&\ \Theta\left((\eps\cdot rs)^2/\delta\right).
\end{align*}
\end{prop}

This shows that the threshold in Theorem \ref{thm:main} is sharp in the regime $2\eps_1\leq\eps_2\leq1/2$. Indeed, taking $r,s\geq\eps_1^{-2}$ and $X=AU(\eps_2,r,s)$ gives

 $$
\dfrac{|\mathcal{N}_{\eps_1}(X)|}{|\mathcal{A}_{\eps_2}(X)|} = \frac{\Theta\left((\eps_1\cdot rs)^2/\eps_2\right)}{\Theta\left(\sqrt{\eps_2}\cdot r^2s^2\right)} = \Theta(  \eps_1^2\cdot \eps_2^{-3/2}).$$

Before proving Proposition \ref{ref:annularconstr}, we introduce a useful lemma for counting the points in an annulus sector. 

\begin{dfn}
    Given $-1\leq \alpha_1<\alpha_2<1$ and $0\leq \delta_1\leq \delta_2<1$, we define an annulus sector associated to these parameters as follows:
    $$AS(\alpha_1,\alpha_2,\delta_1,\delta_2):=\left\{ \frac{1-t}{2}\cdot \big( \cos\left(\pi \alpha\right), \sin\left(\pi\alpha\right)\big) : (\alpha,t)\in [\alpha_1,\alpha_2]\times [\delta_1,\delta_2]\right\}.$$
\end{dfn}

\begin{lem}\label{lem:annulus-count}
        Let $-1\leq \alpha_1<\alpha_2<1$, $0\leq \delta_1\leq \delta_2< \delta<1$ and $r,s\geq 2$ be integers. Then: 
        $$\big|AU(\delta,r,s)\cap AS(\alpha_1,\alpha_2,\delta_1,\delta_2)\big|=\big(1+\lfloor \alpha_2r\rfloor-\lceil \alpha_1r\rceil\big)\cdot\big(1+\lfloor \delta_2s/\delta\rfloor-\lceil  \delta_1s/\delta\rceil\big).$$
\end{lem}

\begin{proof}
 Consider the point $P\in AU(\delta,r,s)$ that has the form
$(1/2-\delta l/2s)\big(\cos(\pi k/r),\sin(\pi k/r)\big)$
with $(k,l)\in\mathbb{Z}[-r,r-1]\times\mathbb{Z}_s$.
Then $P$ lies in $AS(\alpha_1,\alpha_2,\delta_1,\delta_2)$
if and only if
$\pi\alpha_1\leq\pi k/r\leq\pi\alpha_2$ and $(1-\delta_2)/2\leq1/2-\delta l/2s\leq(1-\delta_1)/2$.
After rearranging, these two inequalities become
$\alpha_1r\leq k\leq\alpha_2r$ and $\delta_1s/\delta\leq l\leq\delta_2s/\delta$, respectively. As $-1\leq\alpha_1<\alpha_2<1$ and
$0\leq\delta_1\leq\delta_2<\delta$,
the number of admissible coordinates is therefore
$1+\lfloor\alpha_2r\rfloor-\lceil\alpha_1r\rceil$
and
$1+\lfloor\delta_2s/\delta\rfloor-\lceil\delta_1s/\delta\rceil$,
respectively. Multiplying these two quantities gives the result. 
\end{proof}


\begin{proof}[Proof of Proposition \ref{ref:annularconstr}]
Set $X:=AU(\delta,r,s)$ and put $\delta_*:=\delta(s-1)/s<\delta$. Note $|X|=2rs$, the diameter of the set $X$ is $1$ and all its points have radial parameter in $[0,\delta_*]$.
We start with estimating $|\mathcal{N}_\eps(X)|$. We claim that for every point $Q\in X$ the ball of radius $\eps$ centered at $Q$ contains $\Theta\left(\eps^2rs/\delta\right)$ other points of $X$. Since each of these points forms an $\eps$-neighboring pair with $Q$, summing over all $Q$ does imply that $|\mathcal{N}_\eps(X)|=\Theta\left(\eps^2r^2s^2/\delta\right)$.

Due to the circular symmetry, we can assume that $Q$ has coordinates $\big((s-\delta q)/2s,0\big)$ for some $q\in \mathbb{Z}_s$. Let $(k,l)\in  \mathbb{Z}[-r,r-1] \times \mathbb{Z}_s$ and let $P\in X$ be a point of coordinates
$(1/2-\delta l/2s)\big(\cos(\pi k/r),\sin(\pi k/r)\big)$. Then the distance $|PQ|$ satisfies: 
    \begin{align*}
        |PQ|^2 &= \left(\dfrac{1}{2}-\dfrac{\delta q}{2s}-\cos\left(\dfrac{\pi k}{r}\right)\cdot \left(\dfrac{1}{2}-\dfrac{\delta l}{2s}\right)\right)^2 +\left(\dfrac{1}{2}-\dfrac{\delta l}{2s}\right)^2\cdot \sin^2\left(\dfrac{\pi k}{r}\right) \\
         &=  \left(\dfrac{1}{2}-\dfrac{\delta q}{2s}\right)^2 - 2\cos\left(\dfrac{\pi k}{r}\right)\cdot \left(\dfrac{1}{2}-\dfrac{\delta q}{2s}\right)\left(\dfrac{1}{2}-\dfrac{\delta l}{2s}\right) + \left(\dfrac{1}{2}-\dfrac{\delta l}{2s}\right)^2\\
         &= \dfrac{\delta^2}{4s^2}\cdot (q-l)^2 + 2\left(1-\cos\left(\dfrac{\pi k}{r}\right)\right)\cdot \left(\dfrac{1}{2}-\dfrac{\delta q}{2s}\right)\left(\dfrac{1}{2}-\dfrac{\delta l}{2s}\right).
    \end{align*}
   \indent Let us note that $\dfrac{1}{2}\geq 2\left(\dfrac{1}{2}-\dfrac{\delta q}{2s}\right)\left(\dfrac{1}{2}-\dfrac{\delta l}{2s}\right)\geq\dfrac{1}{2}(1-\delta)^2\geq\dfrac{1}{8}$ since $1-\delta\geq \dfrac{1}{2}$. We will also use that $2x^2/\pi^2\leq 1-\cos x\leq x^2/2$, valid whenever $|x|\leq \pi$.
 
 Consequently, if $|k|> 2\varepsilon r$, then $|PQ|^2\geq (1-\cos(\pi k/r))/8\geq (k/2r)^2>\varepsilon^2$. Similarly, if $|q-l|>2\varepsilon s/\delta$, then $|PQ|^2>\varepsilon^2$. Setting $\tau:=\delta q/s$ and defining
$[a,b]:=[\tau-2\eps,\tau+2\eps]\cap[0,\delta_*]$, we see that the ball of radius $\eps$ centered at $Q$ is contained in $AS(-2\eps,2\eps,a,b)$, where $b-a\leq4\eps$.
 
For the lower bound, pick $c\in[0,\delta_*-\eps]$ such that
$\tau\in[c,c+\eps]$. Such a choice is possible since
$\delta_*\geq\delta/2\geq\eps$.
If $|q-l|\leq\eps s/\delta$ and $|k|\leq\eps r/2$, then
$|PQ|^2\leq\dfrac{\eps^2}{4}+\dfrac{\pi^2k^2}{4r^2}
\leq\left(1+\dfrac{\pi^2}{4}\right)\dfrac{\eps^2}{4}<\eps^2$.
Consequently, all points of 
$X\cap AS(-\eps/2,\eps/2,c,c+\eps)$ lie in the ball of radius $\eps$ centered at $Q$.

Both these annular sectors satisfy the hypotheses of Lemma \ref{lem:annulus-count}, which implies that every $Q$ forms $\Theta(\eps^2rs/\delta)$ neighboring pairs, as claimed.
We now move to $\mathcal{A}_\delta(X)$. Similarly as before, we claim that for every point $Q\in X$ there are $\Theta\left(\sqrt{\delta}\cdot rs\right)$ points in $X$ which are $\delta$-antipodal with $Q$. Again, as $|X|=2rs$, this claim implies that $|\mathcal{A}_\delta(X)|=\Theta\left(\sqrt{\delta}\cdot r^2s^2\right)$.

We may assume this time that $Q$ has coordinates $\big((\delta q-s)/2s,0\big)$ for some $q\in \mathbb{Z}_s$, i.e. it is situated on the negative part of the $x$-axis due to the symmetry. We let $(k,l)\in \mathbb{Z}[-r,r-1]\times \mathbb{Z}_s$ and let $P\in AU(\delta,r,s)$ be the point with coordinates
$(1/2-\delta l/2s)\big(\cos(\pi k/r),\sin(\pi k/r)\big)$. We estimate $|PQ|$ in a similar fashion: 
 \begin{align*}
        |PQ|^2 &= \left(\dfrac{1}{2}-\dfrac{\delta q}{2s}+\cos\left(\dfrac{\pi k}{r}\right)\cdot \left(\dfrac{1}{2}-\dfrac{\delta l}{2s}\right)\right)^2 +\left(\dfrac{1}{2}-\dfrac{\delta l}{2s}\right)^2\cdot \sin^2\left(\dfrac{\pi k}{r}\right) \\
         &=  \left(\dfrac{1}{2}-\dfrac{\delta q}{2s}\right)^2 + 2\cos\left(\dfrac{\pi k}{r}\right)\cdot \left(\dfrac{1}{2}-\dfrac{\delta q}{2s}\right)\left(\dfrac{1}{2}-\dfrac{\delta l}{2s}\right) + \left(\dfrac{1}{2}-\dfrac{\delta l}{2s}\right)^2\\
         &= \left(1-\delta\cdot \dfrac{q+l}{2s}\right)^2  - 2\left(1-\cos\left(\dfrac{\pi k}{r}\right)\right)\cdot \left(\dfrac{1}{2}-\dfrac{\delta q}{2s}\right)\left(\dfrac{1}{2}-\dfrac{\delta l}{2s}\right).
    \end{align*}
  \indent Using the preceding bounds, we obtain that $|PQ|^2\leq 1-\big(1-\cos(\pi k/r)\big)/8\leq 1-k^2/4r^2$. Therefore, if $|k|>\sqrt{8\delta}\,r$, then
$|PQ|^2<1-2\delta<(1-\delta)^2$. For $\delta<1/16$, all points $P$ $\delta$-antipodal with $Q$ thus lie in $AS\left(-\sqrt{8\delta},\sqrt{8\delta},0,\delta_*\right)$,
which contains $O(\sqrt{\delta}\cdot rs)$ points by Lemma \ref{lem:annulus-count}. For $1/16\leq\delta\leq1/2$, the same upper bound follows directly from $|X|=2rs$.

 Finally, we claim that if $2l\leq s$ and $|k|\leq \sqrt{2\delta}\cdot r/\pi$, then $|PQ|^2\geq (1-\delta)^2$, which would imply that all points of $X\cap AS(-\sqrt{2\delta}/\pi,\sqrt{2\delta}/\pi,0,\delta/2)$ are $\delta$-antipodal with $Q$. Consequently, Lemma \ref{lem:annulus-count} applied to these two annulus sectors would then give us that there are $\Theta\left(\sqrt{\delta}\cdot rs\right)$ such points, as desired.\\
 \indent To prove this last claim, we have $\left(1-\delta\cdot \dfrac{q+l}{2s}\right)^2\ge 4\left(\dfrac{1}{2}-\dfrac{\delta q}{2s}\right)\left(\dfrac{1}{2}-\dfrac{\delta l}{2s}\right)$ by AM-GM,
 while also $2\big(1-\cos(\pi k/r)\big)\leq \pi^2k^2/r^2$. It follows that
 $$|PQ|^2\geq \left(4-\dfrac{\pi^2k^2}{r^2}\right)\left(\dfrac{1}{2}-\dfrac{\delta q}{2s}\right)\left(\dfrac{1}{2}-\dfrac{\delta l}{2s}\right)\geq \left(1-\dfrac{\pi^2k^2}{4r^2}\right)\left(1-\dfrac{\delta}{2}\right)(1-\delta).$$
  \indent Therefore, it is enough to prove that $1-(\pi k/2r)^2\geq (1-\delta)(1-\delta/2)^{-1}=1-\delta(2-\delta)^{-1}$. But this is true since $(\pi k/2r)^2\leq (\sqrt{2\delta}/2)^2=\delta/2<\delta(2-\delta)^{-1}$, finishing the proof.
\end{proof}

\subsection{Reuleaux triangle constructions}

In this subsection we study an extremal construction which is based on the Reuleaux triangle. Suppose $P_0,P_1,P_2$ are points in $\mathbb{R}^2$ forming an equilateral triangle of side length $1$. We draw, for each $i\in \{0,1,2\}$, an arc between $P_i$ and $P_{i+1}$ centered at $P_{i+2}$, where the indices are understood modulo $3$. The resulting shape is called the \emph{Reuleaux triangle}. It has a constant width of 1 across all directions and a diameter of 1.

\begin{exam}[Reuleaux construction]
    Let $0<\eps<1$ and let points $P_0,P_1,P_2$ be the corners of a Reuleaux triangle of diameter $1$. For each $i\in \{0,1,2\}$, consider the points $A_i, B_i,C_i,D_i$ on the arc $P_{i+1}P_{i+2}$ such that $|A_iP_{i+1}|=|B_iP_{i+2}|=1/3$ and $|C_iP_{i+1}|=|D_iP_{i+2}|=\eps/2$, where all the indices are modulo $3$. We define the set $RT(\eps,x,y)$ in the following way: for each $i\in\{0,1,2\}$ we place $y$ equally spaced points along each of the arcs $A_iB_i$, then we place $x$ points inside the triangle $\mathcal{T}_i:=P_iC_{i-1}D_{i+1}$, so that no two  are at distance bigger than $\varepsilon$.
    (See Figure \ref{fig:reucon1})  
\end{exam}

For large enough values of $x$ and $y$, we can sharply determine $\mathcal{N}_{\eps_1}(X)$ and $\mathcal{A}_{\eps_2}(X)$ for the set $X:=RT(\eps_2,x,y)$ whenever $2\eps_1\geq \eps_2$, as illustrated by the next result.

\begin{prop}\label{2nd-example}
    Let $0<\eps_2\leq2\eps_1<3^{-3}$ and let $x,y$ be large enough integers. Then:
    \begin{align*}
     \left|\mathcal{A}_{\eps_2}\big(RT(\eps_2,x,y)\big)\right|=&\ 3x^2+3xy; \\
     \left|\mathcal{N}_{\eps_1}\big(RT(\eps_2,x,y)\big)\right|=&\ \Theta\big(\eps_1\cdot y^2+ x^2\big).
\end{align*}
\end{prop}

Setting $y:=tx$ for some parameter $t\geq 1$ we get $\left|\mathcal{N}_{\eps_1}\big(RT(\eps_2,x,y)\big)\right|=\Theta\left(x^2(1+t^2\eps_1)\right)$ and $\left|\mathcal{A}_{\eps_2}\big(RT(\eps_2,x,y)\big)\right|=\Theta (tx^2)$ above. By picking $t=\Theta\left(\eps_1^{-1/2}\right)$, it is now easy to observe that we achieve the $\sqrt{\eps_1}$ threshold from Theorem \ref{thm:main} for the ratio $|\mathcal{N}_{\eps_1}(X)|/|\mathcal{A}_{\eps_2}(X)|$.

\begin{proof}[Proof of Proposition \ref{2nd-example}.]
    Set $X:=RT(\eps_2,x,y)$ and  start with $\mathcal{A}_{\eps_2}(X)$. We claim that all pairs from $\mathcal{A}_{\eps_2}(X)$ must have one of their points in some $\mathcal{T}_i$. Indeed, notice that the longest distance that can occur between two points of $X$ lying on the three arcs $A_iB_i$ is given by $|A_iB_{i+1}|$. These three segments are congruent, so let us consider $|A_1B_2|$. For the convex quadrilateral $A_1B_2P_1P_2$, we have $|A_1P_2|\cdot |B_2P_1|=3^{-2}>3\eps_2$, hence we infer from Lemma \ref{quadrilateral} that $|A_1B_2|<1-\eps_2$.

\begin{figure}[!ht]
    \centering
    \includegraphics[width=0.6\linewidth]{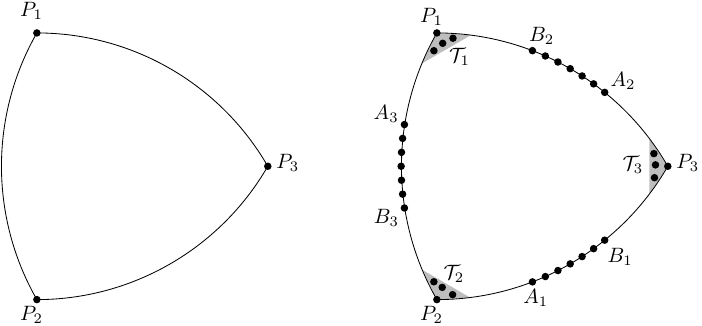}
    \caption{The Reuleaux triangle and the first construction based on it.}
    \label{fig:reucon1}
\end{figure}
    
    Next, let us count how many $\eps_2$-antipodal pairs a point $Q\in X\cap \mathcal{T}_1$ forms in $X$. 
    We start by noting that the longest distance from $Q$ to any point of $X$ lying on one of the arcs $A_2B_2$ and $A_3B_3$ is given by $|QA_{2}|$ and $|QB_{3}|$, respectively. However, the angle $\angle QA_{2}P_{3}$ is obtuse, hence $|QA_{2}|^2<|QP_{3}|^2-|A_{2}P_{3}|^2\leq 1-3\eps_2\leq 1-2\eps_2-\eps_2^2$, which gives $|QA_{2}|<1-\eps_2$. Similarly, 
    $|QB_{3}|<1-\eps_2$, implying that only the points from inside $\mathcal{T}_{2},\mathcal{T}_{3}$ and those on the arc $A_1B_1$ form $\eps_2$-antipodal pairs with $Q$. We claim that all of them do. Indeed, for any point $C$ on the arc $A_1B_1$ we have $|CQ|\geq|P_1C|-|P_1Q|\geq 1-\eps_2/2$, while for any point $D$ inside $\mathcal{T}_j$ with $j\neq 1$ we have $|QD|\geq |P_1P_j|-|P_1Q|-|P_jD|\geq 1-\eps_2$. Hence $Q$ appears in $2x+y$ antipodal pairs. The same argument works for $Q\in X\cap \big(\mathcal{T}_2\cup \mathcal{T}_3\big)$. In total we obtain $|\mathcal{A}_{\eps_2}(X)|=3x^2+3xy$.   
    
   Finally, we shift our focus to $\mathcal{N}_{\eps_1}(X)$. Clearly, if two points are $\eps_1$-neighboring, then they must lie inside the same triangle $\mathcal{T}_i$ or on the same arc $A_iB_i$. Note that if $y$ is large enough, then each point from one of the arcs $A_iB_i$ forms $\Theta(\eps_1y)$ $\eps_1$-neighboring pairs within $X$. Hence $\Theta(\eps_1y^2)$ neighboring pairs come from the arcs. Moreover, any pair of points inside the same $\mathcal{T}_i$ is $\eps_1$-neighboring and thus $3x(x-1)/2$ other $\eps_1$-neighboring pairs get added to $\mathcal{N}_{\eps_1}(X)$. This yields the desired bounds on $|\mathcal{N}_{\eps_1}(X)|$, finishing the proof.
\end{proof}

\medskip

We conclude this subsection with another construction of diameter $1$ based on the Reuleaux triangle. Its purpose is to show that the degree tail bound for the box graph from Lemma~\ref{lem:tails}, which is a central ingredient to the proof, is tight up to a constant factor.

\begin{exam}[Sharpness of tail lemma]\label{ex:reuleaux-tail}
Let $0<2\eps_1\leq\eps_2\leq0.01$ and let $\mathcal{R}$ denote the Reuleaux triangle determined by the points $R=(1/2,\sqrt{3}/2), P_0=(0,0), Q_0=(1,0)$.
Set $m:=\lfloor1/(4\eps_1)\rfloor$ and $n:=\lfloor\eps_2/(4\eps_1)\rfloor$.
For $k\in\mathbb{Z}_{m+1}$ and $l\in\mathbb{Z}_{n+1}$, define the points:
\begin{align*}
P_{k,l}&:=\left(1-\left(1-\frac{l\eps_1}{2}\right)\cos\frac{k\eps_1}{2},\
\left(1-\frac{l\eps_1}{2}\right)\sin\frac{k\eps_1}{2}\right),\\
Q_{k,l}&:=\left(\left(1-\frac{l\eps_1}{2}\right)\cos\frac{k\eps_1}{2},\
\left(1-\frac{l\eps_1}{2}\right)\sin\frac{k\eps_1}{2}\right).
\end{align*}
We then define $RA(\varepsilon_1,\varepsilon_2):=\{R\}\cup\{P_{k,l},Q_{k,l}:(k,l)\in\mathbb{Z}_{m+1}\times\mathbb{Z}_{n+1}\}$. (See Figure~\ref{fig:reuleaux-tail})
\end{exam}

\begin{figure}[ht]
\centering
\begin{tikzpicture}[scale=5.0,line width=0.45pt,font=\small]
  \path[fill=gray!12]
    plot[domain=0:42.9718,samples=45] ({sin(\x)},{-1+cos(\x)})
    -- plot[domain=42.9718:0,samples=45] ({0.94*sin(\x)},{-1+0.94*cos(\x)}) -- cycle;
  \path[fill=gray!12]
    plot[domain=0:42.9718,samples=45] ({sin(\x)},{-cos(\x)})
    -- plot[domain=42.9718:0,samples=45] ({0.94*sin(\x)},{-0.94*cos(\x)}) -- cycle;
  \path[fill=gray!45]
    plot[domain=0:17.1887,samples=30] ({sin(\x)},{-1+cos(\x)})
    -- plot[domain=17.1887:0,samples=30] ({0.94*sin(\x)},{-1+0.94*cos(\x)}) -- cycle;
  \path[fill=gray!45]
    plot[domain=0:17.1887,samples=30] ({sin(\x)},{-cos(\x)})
    -- plot[domain=17.1887:0,samples=30] ({0.94*sin(\x)},{-0.94*cos(\x)}) -- cycle;
  \draw[gray!70,line width=0.3pt]
    plot[domain=0:42.9718,samples=45] ({0.94*sin(\x)},{-1+0.94*cos(\x)});
  \draw[gray!70,line width=0.3pt]
    plot[domain=0:42.9718,samples=45] ({0.94*sin(\x)},{-0.94*cos(\x)});
  \draw (0,0) arc (90:30:1) arc (-30:-90:1) arc (210:150:1);
  \foreach \ang in {0,8.5944,17.1887,25.7831,34.3775,42.9718}{
    \foreach \rho in {1,0.97,0.94}{
      \fill ({\rho*sin(\ang)},{-1+\rho*cos(\ang)}) circle[radius=0.18pt];
      \fill ({\rho*sin(\ang)},{-\rho*cos(\ang)}) circle[radius=0.18pt];
    }
  }
  \fill ({sqrt(3)/2},-0.5) circle[radius=0.22pt];
  \node[above left] at (0,0) {$P_0$};
  \node[below left] at (0,-1) {$Q_0$};
  \node[right] at ({sqrt(3)/2},-0.5) {$R$};
  \node[left] at (-0.025,-0.03) {$v=P_{0,l}$};
  \node at (0.17,-0.155) {$S_1$};
  \draw[gray!70,line width=0.3pt] (0.17,-0.125) -- (0.17,-0.05);
  \node at (0.17,-0.845) {$S_2$};
  \draw[gray!70,line width=0.3pt] (0.17,-0.875) -- (0.17,-0.95);
  \draw[dashed,line width=0.35pt]
    ({0.94*sin(17.1887)},{-1+0.94*cos(17.1887)})
    -- ({0.94*sin(17.1887)},{-0.94*cos(17.1887)});
  \node[right] at (0.285,-0.5) {$P_{k_1,l_1}Q_{k_2,l_2}$};
\end{tikzpicture}
\caption{The last Reuleaux construction and the sets $S_1,S_2$ used in the proof.}
\label{fig:reuleaux-tail}
\end{figure}

\begin{prop}\label{prop:tail-sharpness}
Let $0<2\eps_1\leq\eps_2\leq0.01$ and let $X:=RA(\varepsilon_1,\varepsilon_2)$. Then there is a box graph $G:=G_X(\eps_1,\eps_2)$ associated to $X$, with $N:=|V(G)|=\Theta(\eps_2\eps_1^{-2})$, having at least $\eps_2/(2\eps_1)$ of its vertices $v$ satisfy, for every integer $t$ with $\eps_2^2\eps_1^{-2}\leq t\leq\eps_2\eps_1^{-2}/16$, the following:
\[
\big|\{w\in V(G):w\sim v\text{ and }\deg_G(w)\geq t\}\big|
\geq\frac{\eps_2^3\eps_1^{-4}}{16t}.
\]
\end{prop}

There is a natural  reason for restricting the range of $t$ in the statement above: the claimed lower bound is proportional to $t^{-1}$, however, for very small $t$ it would exceed the total number of vertices, while for $t$ above the maximum degree of the box graph, the tail is zero.

\begin{proof}[Proof of Proposition \ref{prop:tail-sharpness}]
We first construct the box graph. Note that $0\leq k\eps_1, l\eps_1\eps_2^{-1}\leq1/4$.
The points $P_{k,0}$ lie on the arc of $\mathcal{R}$ centered at $Q_0$, and $P_{k,l}$ lies on the segment $P_{k,0}Q_0$. Similarly, $Q_{k,l}$ lies on $Q_{k,0}P_0$. By convexity, all points of $RA(\varepsilon_1,\varepsilon_2)$ belong to or lie inside $\mathcal{R}$.

Our first claim is that $X$ is $\eps_1/4$-separated. Indeed, for some $P_{k,l}$ and $P_{k',l'}$, we have
$$|P_{k,l}P_{k',l'}|^2
=\frac{(l-l')^2\cdot\eps_1^2}{4}
+2\left(1-\frac{l\eps_1}{2}\right)\left(1-\frac{l'\eps_1}{2}\right)
\left(1-\cos\frac{(k-k')\eps_1}{2}\right).$$
\indent If $l\neq l'$, their distance is at least $\eps_1/2$. If $k\neq k'$, it is at least
$2(1-\eps_2/8)\sin(\eps_1/4)>\eps_1/4$.
The same argument applies to points $Q_{k,l}$. Moreover, the first and second families of points lie within distance $1/8+\eps_2/8$ of $P_0$ and $Q_0$, respectively. Hence points from different families are at distance greater than $1/2$, and the same is true for the distance from $R$ to either family.

Let $K:=\text{conv}(X)$ and define $K_{\eps_2}$ as in Section \ref{sec:reduction}. It is easy to check that $X\subseteq K_{\eps_2}$ and it can be extended to a maximal $\eps_1/4$-separated set $S$ inside $K_{\eps_2}$.
The squares of side length $\eps_1/2$ centered at the points of $S$ give the vertices of our box graph $G:=G_X(\eps_1,\eps_2)$.
We may also identify each point of $X$ with the vertex whose box is centered at that point. In particular, an $\eps_2$-antipodal pair from $X$ certifies an edge between the corresponding vertices. This implies that $N\geq|X|=2(m+1)(n+1)+1\geq2\cdot(4\eps_1)^{-1}\cdot\eps_2/4\eps_1=\eps_2\eps_1^{-2}/8$. On the other hand, the packing estimate from Section \ref{sec:reduction} gives $N=O(\eps_2\eps_1^{-2})$, hence $N=\Theta(\eps_2\eps_1^{-2})$.

Let $t$ be an integer in the stated range and put $a:=2t\varepsilon_1^2\varepsilon_2^{-1}$, $b:=\varepsilon_2(4a)^{-1}$, so that $ab=\eps_2/4$. Our assumptions also imply that $0<a,b\leq1/8$.
Let $S_1:=\{P_{k_1,l}:0\leq k_1\eps_1/2\leq a,\ l\in\mathbb{Z}_{n+1}\}$ and 
$S_2:=\{Q_{k_2,l}:0\leq k_2\eps_1/2\leq b,\ l\in\mathbb{Z}_{n+1}\}$. It is not hard to lower bound their sizes:
\begin{align*}
|S_1|
&=(n+1)\big(1+\lfloor2a/\eps_1\rfloor\big)
\geq\frac{\eps_2}{4\eps_1}\cdot\frac{2a}{\eps_1}
=\frac{\eps_2a}{2\eps_1^2}=t,\\
|S_2|
&=(n+1)\big(1+\lfloor2b/\eps_1\rfloor\big)
\geq\frac{\eps_2}{4\eps_1}\cdot\frac{2b}{\eps_1}
=\frac{\eps_2b}{2\eps_1^2}
=\frac{\eps_2^3\eps_1^{-4}}{16t}.
\end{align*}
\indent We next claim that all possible edges between $S_1$ and $S_2$ appear in $G$, i.e. that $S_1$ and $S_2$ induce a complete bipartite subgraph in $G$.
To prove this, select $P_{k_1,l_1}\in S_1$ and $Q_{k_2,l_2}\in S_2$, then set $\phi_1:=k_1\eps_1/2$ and $\phi_2:=k_2\eps_1/2$. We can write:
$$|P_{k_1,0}Q_{k_2,0}|^2=3+2\cos(\phi_1+\phi_2)-2\cos\phi_1-2\cos\phi_2
=1-8\sin(\phi_1/2)\sin(\phi_2/2)\cos((\phi_1+\phi_2)/2).$$
\indent This implies, using $\sin\phi\leq \phi$ and $\cos\phi\leq 1$, that $|P_{k_1,0}Q_{k_2,0}|^2\geq1-2\phi_1\phi_2$. Consequently, since $\phi_1\phi_2\leq ab=\eps_2/4$, we get that
$|P_{k_1,l_1}Q_{k_2,l_2}|
\geq\sqrt{1-2\phi_1\phi_2}-\eps_2/4
\geq1-2\phi_1\phi_2-\eps_2/4>1-\eps_2.$ But this means exactly that there is an edge in $G$ between the boxes of $P_{k_1,l_1}$ and $Q_{k_2,l_2}$.

In particular, every vertex $v=P_{0,l}$ with $l\in \mathbb{Z}_{n+1}$ is adjacent to every vertex of $S_2$. Fix such a vertex and note that 
every vertex of $S_2$ is adjacent to $v$ and has at least $|S_1|\geq t$ neighbors. Thus $S_2\subseteq\{w\sim v:\deg_G(w)\geq t\}$, which gives the desired bound.
The same argument applies to each vertex $Q_{0,l}$ by simply interchanging the two families. 
There are $2(n+1)\geq\eps_2/(2\eps_1)$ such vertices, finishing the proof.
\end{proof}

\subsection{Sharpness over the full parameter range}

In this final subsection we show that the preceding constructions
establish the sharpness of Theorem~\ref{thm:main} for all $0<\eps_1,\eps_2<1$. For convenience, define for a point set $X$: $$\mathcal{R}_X(\eps_1,\eps_2):=\frac{|\mathcal{N}_{\eps_1}(X)|}
     {|\mathcal{A}_{\eps_2}(X)|}.$$

Fix $\delta_0=1/100$ and suppose first that $\eps_1,\eps_2<\delta_0$.
If $2\eps_1\leq\eps_2$, the annulus construction
from Proposition~\ref{ref:annularconstr} gives $|\mathcal{R}_X(\eps_1,\eps_2)| =\Theta(\eps_1^2\eps_2^{-3/2})$. Else, $\eps_2<2\eps_1$, and the Reuleaux
construction from Proposition~\ref{2nd-example}, with $y=\Theta(x\eps_1^{-1/2})$, gives
$|\mathcal{R}_X(\eps_1,\eps_2)| =\Theta(\eps_1^{1/2})$.
When $\eps_1\leq\eps_2<2\eps_1$, we have
$\eps_1^{1/2}
=\Theta(\eps_1^2\eps_2^{-3/2})$,
so this also yields the required bound in that regime.

It remains to consider the case
$\max\{\eps_1,\eps_2\}\geq\delta_0$.
If $\eps_1\leq\delta_0/2$, then 
$\eps_2\geq\delta_0$.
Proposition~\ref{ref:annularconstr}, applied 
with $\delta=\delta_0$ and $\eps=\eps_1$,
and the monotonicity of
$|\mathcal{A}_{\eps_2}(X)|$ in $\eps_2$, help us deduce that
$|\mathcal{R}_X(\eps_1,\eps_2)|=O(\eps_1^2)=O(\eps_1^2\eps_2^{-3/2})$.
Finally, suppose that $\eps_1>\delta_0/2$.
Take two equally sized clusters of distinct points,
arbitrarily close to the two endpoints of a unit
segment, respectively.
Choosing the clusters sufficiently small ensures
that every cross-cluster pair is
$\eps_2$-antipodal.
If each cluster contains $m$ points, then $|\mathcal{A}_{\eps_2}(X)|\geq m^2$ and $|\mathcal{N}_{\eps_1}(X)|\leq m(2m-1)$. We therefore get $|\mathcal{R}_X(\eps_1,\eps_2)|\leq 2$.

Since $\eps_1>\delta_0/2$, both the product
$\eps_1^2\eps_2^{-3/2}$ when
$\eps_1\leq\eps_2$, and the product
$\eps_1^{1/2}$ when $\eps_1>\eps_2$,
are bounded below by positive absolute constants.
Thus the desired upper bound follows in both cases.
This completes the proof of sharpness throughout
the full parameter range.

\section{Concluding remarks and open questions}\label{sec:conclude}

The box graph arises naturally when one seeks to analyze the relationship between antipodal and neighboring pairs. After reducing the problem to estimating the largest eigenvalue of its adjacency matrix, one of the main difficulties was the selection of an appropriate test vector for the Collatz–Wielandt bound in Proposition~\ref{vector-test-criterion}.
Using the all-ones vector, it is easy to get an upper bound of \textcolor{blue}{$\varepsilon$}. However, a box graph may contain vertices of very large degree, potentially linear in $N$, which renders the all-ones vector unsuitable. A more refined candidate is the degree vector $\mathbf{d}$, since $A\mathbf{d}$   counts the number of $2$-paths originating at each vertex, where $A$ denotes the adjacency matrix of $G$. In the one-parameter case, estimating the number of $2$-paths first only yields an $O(\sqrt{\eps^{-1}\log(\eps^{-1})})$ upper bound for the largest eigenvalue, leading to the threshold obtained by Korsky \cite{korsky2026optimalspectralboundsantipodal}.
In the examples we investigated, the vertices which generate a large number of $2$-paths appear to be relatively sparse. 
This suggests averaging over longer paths as a way to mitigate this concentration phenomenon. Indeed, our calculations suggested that by looking at $A^2\mathbf{d}$ --  which corresponds to counting $3$-paths -- we might be able to get a sharp result. Finally, we arrived at the vector $(\sqrt{d_1},\sqrt{d_2},\ldots,\sqrt{d_N} )$ by noting that it works for all examples known to us. \medskip

The problem is also worth investigating in higher dimensions. As illustrated by the annulus construction (Example 1), the extremal behavior in $\mathbb{R}^2$ of the ratio between $\mathcal{N}_{\eps_1}(X)$ and $\mathcal{A}_{\eps_2}(X)$ in the regime $\varepsilon_1 < \varepsilon_2$ can be achieved by comparing the area of a ball of radius $\varepsilon_1/2$ with that of the subset of a ball of radius $1/2$ consisting of points at distance at least $1-\varepsilon_2$ from a fixed point located near the boundary. In contrast, when $\varepsilon_1 > \varepsilon_2$, the corresponding threshold appears to depend solely on $\varepsilon_1$. We conjecture that an analogous phenomenon persists in higher dimensions: for $d \geq 3$, one of the extremal configurations should be realized by a set of points distributed uniformly in a region $\mathcal{R}$ of $\mathbb{R}^d$ defined as the space between two concentric balls $B_1$ and $B_2$ of diameters $1$ and $1 - \varepsilon_2$, respectively. 

A ball of radius $\eps_1/2$ has volume of order $\Theta_d\big(\eps_1^d\big)$, and if we fix a point $Q$ in $\mathcal{R}$, it can be shown that the subset of $\mathcal{R}$ made of points which are $\eps_2$-antipodal with $Q$ has volume $\Theta_d\big(\eps_2^{(d+1)/2}\big)$. When $\eps_1<\eps_2$, plenty of balls of radius $\eps_1/2$ can be fitted inside $\mathcal{R}$, leading to the volume ratio threshold. When $\eps_1>\eps_2$ this is no longer the case, and the behavior is the same as for $\eps_1=\eps_2$, since extremal constructions can be pushed closer to the boundary of $B_1$. 

Motivated by these observations, we propose the following conjecture.

\begin{conjecture}\label{conj:general}
    For each fixed $d\geq3$, there are constants $c_1,c_2>0$, depending only on $d$, such that for any $0<\varepsilon_1,\varepsilon_2<1$ and any finite set $X\subset \mathbb{R}^d$ of diameter at most $1$ with $|X|\ge c_1\varepsilon_1^{-d}$ the following holds: 
\[
\dfrac{|\mathcal{N}_{\eps_1}(X)|}{|\mathcal{A}_{\eps_2}(X)|} \ge
\begin{cases}
c_2\cdot \eps_1^d\cdot \eps_2^{-(d+1)/2}, & \eps_1 \le \eps_2; \\
 c_2\cdot \eps_1^{(d-1)/2}, & \eps_1 > \eps_2.
\end{cases}
\]
\end{conjecture}

The reduction to determining the largest eigenvalue of the box graph naturally carries over to higher dimensions. Thus, extending our approach to $\mathbb{R}^d$ amounts to finding a suitable analog of Lemma \ref{quadrilateral} and understanding the corresponding tail bounds it would imply on the degree distribution of the box graph.

We note that during the refereeing process of this paper Korsky published a preprint where he claims that the $\varepsilon_1=\varepsilon_2$ case of Conjecture \ref{conj:general} holds in any dimension $d\ge 1$, see \cite{korsky2026d-dim}. 


Another natural direction is to study the problem in different metric spaces. Already in $\mathbb{R}^2$, the answer depends on the choice of norm. 
For instance, consider the $\ell_\infty$ norm, and place $n/2$ points evenly along the segment from $(0,0)$ to $(0,1)$, while the remaining $n/2$ points are evenly placed along the segment from $(1,0)$ to $(1,1)$. Then there are $n^2/4$ pairs at distance exactly $1$ and $\Theta(n^2 \varepsilon)$ $\varepsilon$-neighboring pairs. This shows the lower bound of order $\varepsilon^{1/2}$ is not attainable in this setting. In fact the order $\eps$ is sharp: using a maximal $\eps/4$-separated set in the $\ell_\infty$ metric, the adjacent centers lie in a square shell of area $O(\eps)$. Disjoint squares of side $\eps/4$ give maximum degree $O(\eps^{-1})$, so the same spectral reduction gives $|\mathcal{N}_\eps(X)|\geq c\eps|\mathcal{A}_\eps(X)|$ for $|X|\geq C\eps^{-2}$.

\section*{Acknowledgment}
We thank J\'anos Pach and Zolt\'an L\'or\'ant Nagy for helpful discussions on this problem. We also thank the anonymous referees for their comments and suggestions, particularly concerning the estimates for the number of vertices of the box graphs. The use of Lemma \ref{lem:tails} was suggested by ChatGPT 5.4; its proof was written by the authors, and revised according to the referees' suggestions. OpenAI Codex assisted with checking the computations and with a revision of the construction underlying Proposition \ref{prop:tail-sharpness} and its proof.

G\'abor Dam\'asdi was supported by the ERC Advanced Grant no. 882971 GeoScape, by the J\'anos Bolyai Research Scholarship of the Hungarian Academy of Sciences, the STARTING-25 project no. 152199 and the EXCELLENCE-24 project no. 151504 Combinatorics and Geometry of the NRDI Fund.

Lauren\c{t}iu Ploscaru was supported by the ERC Advanced Grants no. 882971 GeoScape and ERMiD no. 101054936.

\printbibliography

\end{document}